\documentclass[10pt]{amsart}
\usepackage{amsmath,amssymb,verbatim}
\usepackage[latin1]{inputenc}
\usepackage{enumerate}
\usepackage{pdflscape}
\usepackage{amsthm}
\usepackage{mathrsfs}
\usepackage[dvips]{graphicx}
\usepackage{color}
\usepackage[normalem]{ulem}
\usepackage{cancel}
\usepackage{tikz}
\usetikzlibrary{matrix}
\usepackage[all]{xy}
\usepackage{mathtools}

\newtheorem{theorem}{Theorem}[section]
\newtheorem{lemma}[theorem]{Lemma}
\newtheorem{proposition}[theorem]{Proposition}

\newtheorem{corollary}[theorem]{Corollary}
\newtheorem{remark}[theorem]{Remark}

\newtheorem{notation}[theorem]{Notation}

\newenvironment{proofof}{\par\noindent \textit{Proof of }}{\qed\par\bigskip}

\newcommand{\Gal}{{\rm Gal}}

\newcommand{\Cen}{{\rm C}}

\newcommand{\Z}{{\mathbb Z}}
\newcommand{\Q}{{\mathbb Q}}

\newcommand{\C}{{\mathbb C}}

\newcommand{\res}{\operatorname{res}}
\newcommand{\GL}{{\rm GL}}

\newcommand{\tr}{{\rm tr}}
\newcommand{\Soc}{{\rm Soc}}
\newcommand{\ind}{{\rm ind}}

\newcommand{\matriz}[1]{\begin{array} #1 \end{array}}

\newcommand{\GEN}[1]{\langle #1 \rangle}

\newcommand{\V}{\mathrm{V}}

\newcommand{\pa}[2]{\varepsilon_{#2}(#1)}
\newcommand{\x}{\widehat{x}}

\newcommand{\qand}{\quad \text{and} \quad}

\DeclareMathOperator{\Cyc}{\rm{Cyc}}
\DeclareMathOperator{\Cocyc}{\rm{Cocyc}}

\begin{document}

\title{Cliff-Weiss Inequalities and the Zassenhaus Conjecture}

\author{Leo Margolis and Ángel del Río}
\thanks{This research is partially supported by the European Commission under Grant 705112-ZC, by the Spanish Government under Grant MTM2016-77445-P with "Fondos FEDER" and, by Fundación Séneca of Murcia under Grant 19880/GERM/15.}

\keywords{Integral group ring, groups of units, Zassenhaus conjecture}

\subjclass[2010]{16U60, 16S34, 20C05, 20C10}

\begin{abstract}
Let $N$ be a nilpotent normal subgroup of the finite group $G$. Assume that $u$ is a unit of finite order in the integral group ring $\Z G$ of $G$ which maps to the identity under the linear extension of the natural homomorphism $G \rightarrow G/N$. We show how a result of Cliff and Weiss can be used to derive linear inequalities on the partial augmentations of $u$ and apply this to the study of the Zassenhaus Conjecture. This conjecture states that any unit of finite order in $\Z G$ is conjugate in the rational group algebra of $G$ to an element in $\pm G$.
\end{abstract}

%\date{}
\maketitle

\section{Introduction}

A conjecture about the torsion units of integral group rings of finite groups has been put forward by H.J. Zassenhaus in \cite{Zassenhaus} and occupied many researchers during the last decades.

\begin{quote}
\textbf{Zassenhaus Conjecture}: Let $G$ be a finite group, $\mathbb{Z}G$ the integral group ring of $G$ and $u$ a unit in $\mathbb{Z}G$ of finite order. Then there exists a unit $x$ in the rational group algebra of $G$ and an element $g \in G$ such that $x^{-1}ux = \pm g$.

\end{quote}

In case such $g$ and $x$ exist one says that $u$ and $\pm g$ are rationally conjugate.
%Originally Zassenhaus stated more conjectures about finite subgroups in the unit group of integral group rings (see e.g. \cite[Section 37]{Sehgal1993}). All of them were proved for finite nilpotent groups \cite{Weiss1991}, but they turned out to be wrong in general \cite{Klingler1991} except the First Zassenhaus Conjecture which is still open and is known today just as the Zassenhaus Conjecture.
Weiss proved the Zassenhaus Conjecture for finite nilpotent groups \cite{Weiss1991}.
Most of the early research on the Zassenhaus Conjecture concentrated on special classes of metabelian groups \cite{PolcinoMiliesSehgal1984, MarciniakRitterSehgalWeiss1987} (see \cite{Hertweck2006} for more references). Almost all of these results were generalized by Hertweck \cite{Hertweck2008}, whose result was generalized in \cite{CaicedoMargolisdelRio2013} stating that the Zassenhaus Conjecture holds for cyclic-by-abelian groups.
Another relevant result established the Zassenhaus Conjecture for groups with a normal Sylow $p$-subgroup with abelian complement \cite{Hertweck2006}.

Denote by $\V(\Z G)$ the group of units of augmentation $1$ in $\mathbb{Z}G$.
Since all the units of $\mathbb{Z}G$ lie in $\pm \V(\Z G)$ it is sufficient to study the elements of $\V(\Z G)$.
Furthermore, for a normal subgroup $N$ in $G$, denote by $\V(\Z G,N)$ the group of units $u$ of $\mathbb{Z}G$ mapped to the identity by the $\Z$-linear extension $\Z G\rightarrow \Z(G/N)$, of the natural homomorphism $G\rightarrow G/N$.
Many of the proofs of the Zassenhaus Conjecture for a group $G$, with a ``nice'' normal subgroup $N$ are divided into two parts: The torsion elements of $\V(\Z G,N)$ are normally studied separately from those in $\V(\Z G)\setminus \V(\Z G,N)$. In this paper we will mostly concentrate on the first case and more precisely we study the following problem.

\begin{quote}
\textbf{Sehgal's Problem} \cite[Research Problem 35]{Sehgal1993}:
Let $N$ be a normal nilpotent subgroup of the finite group $G$. Is every torsion element of $\V(\Z G,N)$  rationally conjugate to an element in $G$?
\end{quote}

In that case we say that Sehgal's Problem has a positive solution for $G$ and $N$.
In case $N$ is a nilpotent group we say that Sehgal's Problem has a positive solution for $N$ if it has a positive solution for every finite group $G$ and every normal subgroup of $G$ isomorphic to $N$.

Weiss' Theorem \cite{Weiss1991} provides a positive solution for Sehgal's Problem if $N=G$ (and hence $G$ is nilpotent).
Moreover, Sehgal's Problem has a positive solution for $N$ if $N$ is a $p$-group \cite[Proposition~4.2]{Hertweck2006}.

Marciniak, Ritter, Sehgal and Weiss proposed a strategy to attack Sehgal's Problem which we will refer to as the matrix strategy. Namely, if $k=[G:N]$ then there is a natural embedding $\Phi:\Z G \rightarrow M_k(\Z N)$, into the $k\times k$-matrix ring over $\mathbb{Z}N$, mapping the elements of $\V(\Z G,N)$ into the group $\V_k(\Z N)$ formed by elements of $\GL_k(\Z N)$ mapped to the identity matrix via the componentwise application of the augmentation map. It can be shown that Sehgal's Problem has a positive solution for $N$ if the following has a positive solution for $N$ (see Section~\ref{SectionSetUp}).

\begin{quote}
\textbf{Matrix Zassenhaus Problem}: Let $N$ be a finite group. Is every element of $\V_k(\Z N)$ conjugate in $\GL_k(\Q N)$ to a diagonal matrix with entries in $N$?
\end{quote}

For $N$ nilpotent, this problem appeared as Research Problem 37 in \cite{Sehgal1993}.
Let $N$ be nilpotent.
Again Weiss' Theorem \cite{Weiss1991} shows that the Matrix Zassenhaus Problem has a positive solution for $k=1$.
See Theorem~\ref{MatrixProblemTh}.\eqref{Problem2Positive} for other positive results on the Matrix Zassenhaus Problem.
During the last years the matrix strategy has been abandoned due to a result of Cliff and Weiss who proved that the Matrix Zassenhaus Problem has a positive solution for $N$ and all the possible values of $k$ if and only if $N$ has at most one non-cyclic Sylow subgroup \cite{CliffWeiss}.
The main aim of this paper is trying to show that this result does not invalid completely the matrix strategy.
Actually we will show that the main character theoretic result on which the Cliff and Weiss Theorem is based provides a method to obtain positive solutions for Sehgal's Problem in cases which go beyond the case where $N$ has at most one non-cyclic Sylow subgroup.
The philosophy behind is that to use the matrix strategy for $G$ and $N$ we do not need to consider all the torsion elements of $\V_k(\Z N)$ but only those which are in $\Phi(\V(\Z G,N))$.
An analysis of the proof of Cliff and Weiss shows that it does not affect these kind of torsion elements (see Subsection~\ref{SubsectionMatrixStrategy} for more details).

Let from here on $G$ always denote a finite group. Denote by $g^G$ the conjugacy class of an element $g \in G$.
The main notion to study the Zassenhaus Conjecture are the so-called partial augmentations.
Let $u=\sum_{g\in G} u_g g$ with $u_g\in \Z$ for every $g\in G$.
For every $g\in G$, the partial augmentation of $u$ at $g$ is
    $$\pa{u}{g^G} = \sum_{x \in g^G} u_x.$$
A torsion  element $u$ of $\V(\Z G)$ is conjugate to an element of $G$ if and only if $\pa{u^d}{g^G} \geq 0$ for all $g \in G$ and all divisors $d$ of the order of $u$ \cite[Theorem 2.5]{MarciniakRitterSehgalWeiss1987}.
Thus obtaining restrictions on the possible partial augmentations of torsion units is a main task in the study of the Zassenhaus Conjecture.
In this paper we obtain restrictions on the partial augmentations of torsion elements $u$ of $\V(\Z G, N)$, where $N$ is a nilpotent normal subgroup of $G$.
They take the form of linear integral inequalities where the variables are the partial augmentations of $u$, cf. Proposition~\ref{EquationCharacter} and Theorem~\ref{Inequalities}.
We call them the Cliff-Weiss inequalities because they are based on a result from \cite{CliffWeiss}.
In the special case where $N$ has an abelian Hall $p'$-subgroup the Cliff-Weiss inequalities take a particularly friendly form.

\begin{proposition}\label{InequalitiesAbelian}
Let $N$ be a nilpotent normal subgroup of $G$ such that $N$ has an abelian Hall $p'$-subgroup $A$ for some prime $p$ and let $u$ be a torsion element of $\V(\Z G, N)$. If $K$ is a subgroup of $A$ such that $A/K$ is cyclic and $n \in N$ then
\[\sum_{g \in nK} |C_G(g)|\pa{u}{g^G} \geq 0. \]
\end{proposition}

The well known HeLP Method provides other integral linear inequalities on the partial augmentations of torsion elements of $\V(\Z G)$.
In \cite[Remark~3.4]{MargolisdelRioPAP} we prove that the Cliff-Weiss inequalities imply the HeLP inequalities and in \cite{MargolisdelRioCWAlgorithm} we show one example where the inequalities from Proposition~\ref{InequalitiesAbelian} prove the Zassenhaus Conjecture for cases where the HeLP Method fails.
Moreover, it is in many situations easier to produce the Cliff-Weiss inequalities compared with the inequalities provided by the HeLP Method.

On the other hand in \cite{MargolisdelRioCWAlgorithm} we have introduced an algorithm to construct candidates to negative solutions to Sehgal's Problem based on Proposition~\ref{InequalitiesAbelian}. 
Eisele and Margolis have proved recently that one of them is indeed a negative solution and hence this has provided the first known counterexample to the Zassenhaus Conjecture \cite{EiseleMargolis}.

In Section~\ref{SectionApplications} we give applications of Proposition~\ref{InequalitiesAbelian} to obtain positive solutions for Sehgal's Problem. These applications can be used under certain restrictions on indices of certain centralizers in $G$, cf. Theorem \ref{Main} and Corollaries~\ref{AtMostOnceNonCyclic}, \ref{OnePrime} and \ref{SimplifiedBound}.
We apply this to obtain new results on the Zassenhaus Conjecture.

\begin{theorem}\label{TheoremZC1}
Let $N$ be a nilpotent normal subgroup of the finite group $G$. Suppose that $[G:N]$ is prime and $N$ has at most one non-cyclic Sylow subgroup.
Then the Zassenhaus Conjecture holds for $G$.
\end{theorem}

\begin{theorem}\label{TheoremZC2}
Let $G$ be a finite group with a nilpotent normal subgroup $N=A\times B$ such that $A$ is an abelian $p$-group, for some prime $p$, $B$ has at most one non-cyclic Sylow subgroup and $[G:N]$ is prime and smaller than $p$.
Then the Zassenhaus Conjecture holds for $G$.
\end{theorem}

More applications can be found in \cite{MargolisdelRioCWAlgorithm} and \cite{MargolisdelRioPAP}.

\medskip

The paper is structured as follows. In Section~\ref{SectionSetUp} we introduce our set up, recall the double action formalism and the matrix strategy of Marciniak, Ritter, Sehgal and Weiss. In Section~\ref{SectionInequalities} we state and prove the Cliff-Weiss inequalities. Section~\ref{SectionLatticeCalculations} contains auxiliar calculations on the lattice of subgroups of a finite abelian group. Finally, in Section~\ref{SectionApplications} we apply the inequalities from Section~\ref{SectionInequalities} and the calculations in Section~\ref{SectionLatticeCalculations} to obtain some positive solutions of Sehgal's Problem which are then used in the proofs of Theorems~\ref{TheoremZC1} and \ref{TheoremZC2}.

\section{Preliminaries}\label{SectionSetUp}

\subsection{Basic notation and a key result of Hertweck}

All throughout $G$ is a finite group. We will denote by $RG$ the group ring of $G$ over a commutative ring $R$.

If $u$ is an element of finite order in a group and $\pi$ is a set of primes then $u_{\pi}$ and $u_{\pi'}$ denote the $\pi$-part and the $\pi'$-part of $u$, respectively.
If $X$ is a subset of a finite group then $X_{\pi}=\{x_{\pi}: x\in X\}$ and $X_{\pi'}=\{x_{\pi'}:x\in X\}$.
In particular, if $N$ is a finite nilpotent group then $N_{\pi}$ and $N_{\pi'}$ denote the Hall $\pi$-subgroup and the Hall $\pi'$-subgroup of $N$.
For a prime $p$ we simplify the notation and write $x_p=x_{\{p\}}$ and $x_{p'}=x_{\{p\}'}$, for $x$ either an element of finite order in a group, a subset of a finite group or a finite nilpotent group.

If $X\subseteq G$ then $C_G(X)$ denotes the centralizer of $X$ in $G$ and for $g \in G$ we set $C_G(g)=C_G(\{g\})$.
Moreover $g^G$ denotes the conjugacy class of $g$ in $G$.
In some cases we are going to use sums running on representatives of conjugacy classes of $G$, not depending on the set of representative chosen. We denote this by writing $\sum_{g^G}$. Furthermore we call elements $g, h \in G$ \textit{locally conjugate} in $G$ if $g_p^G = h_p^G$ for each prime $p$. We denote the set of elements in $G$ locally conjugate to $g$ by $\ell_G(g)$.

We use the standard notation $\res^G_H$ and $\ind^G_H$ for restriction and induction of characters and $M_k$ and $\GL_k$ for $k\times k$ matrix rings and general linear groups, respectively.
For complex class functions $\chi$ and $\psi$ of $G$ denote by $\langle \chi, \psi \rangle_G$ the hermitian product of $\chi$ and $\psi$ in the space of class functions of $G$, i.e. $\GEN{\chi,\psi}_G = \frac{1}{|G|} \sum_{g\in G} \phi(g)\overline{\psi(g)}$.
Let $\tr:M_k(RG)\mapsto RG$ denote the trace map, i.e. $\tr(A)$ is the sum of the diagonal entries of $A$.

One of the main tools of this paper is the following theorem which is a consequence of results of Hertweck (see \cite[Theorem~2]{MargolisHertweck} and \cite[Lemma 2.2]{Hertweck2008}).

\begin{theorem}\label{HertweckPAdic}
Let $G$ be a finite group, let $N$ be a nilpotent normal subgroup of $G$ and let $u$ be a torsion element of $\V(\Z G,N)$.
Then there is $n \in N$  such that
\begin{enumerate}
\item $u_p$ is rationally conjugate to $n_p$ for every prime $p$. This conjugation can be also realized in $\Z_p G$ where $\Z_p$ denotes the ring of $p$-adic integers.
\item If $\pa{u}{g^G} \neq 0$ with $g \in G$ then $g\in \ell_G(n)$ and, in particular, $g \in N$.
\end{enumerate}
\end{theorem}

\subsection{Double action modules}

An idea to study the Zassenhaus Conjecture, and more generally the Matrix Zassenhaus Problem, was introduced in \cite{MarciniakRitterSehgalWeiss1987} and is sometimes called the "double-action formalism".
For a group $U$ and a homomorphism $\alpha: U \rightarrow \GL_k(RG)$ define $R[\alpha]$ as the left $R(U\times G)$-module with underlying additive group
$(RG)^k$ and group action given as follows:
\[(u,g)x=\alpha(u)xg^{-1}, \quad u\in U, \ g\in G, \ x\in R[\alpha],\]
where $x$ is considered as a column matrix of length $k$.
As $R[\alpha]$ is free as $R$-module and $R$ is commutative, associated to $\alpha$ there is an $R$-character $\chi_{\alpha}$.
Two maps $\alpha,\beta:U\rightarrow \GL_k(RG)$ are conjugate in $\GL_k(RG)$ (i.e. there is $V\in \GL_k(RG)$ such that $\alpha(u)=V^{-1} \beta(u) V$ for every $u\in U$) if and only if $R[\alpha]\cong R[\beta]$.
This links the Matrix Zassenhaus Problem with character theoretic questions, a fact which was exploited e.g. in \cite{CliffWeiss} and \cite{Hertweck2008}.

The character $\chi_{\alpha}$ is related with the partial augmentations via the following formula:
\begin{equation}\label{CharacterPA}
\chi_{\alpha}(u,g)=|\Cen_G(g)| \; \pa{\tr(\alpha(u)}{g^G} \ \ \text{for} \ \ u \in U, \ g \in G
\end{equation}
(see \cite[38.12]{Sehgal1993} or \cite[Lemma~1]{Weiss1991}).
Observe that this implies that if $\alpha$ and $\beta$ are conjugate in $\GL_k(RG)$ then $\pa{\tr(\alpha(u))}{g^G}=\pa{\tr(\beta(u))}{g^G}$ for every $u\in U$.
In the case $k = 1$ we thus obtain
\begin{equation}\label{CharacterPA1}
\chi_\alpha(	u,g) = |C_G(g)|\pa{\alpha(u)}{g^G}.
\end{equation}

\subsection{The matrix strategy}\label{SubsectionMatrixStrategy}

Now let $N$ be a normal subgroup in $G$ of index $k$ and consider $RG$ as an $(RN,RG)$-bimodule in the natural way.
As $_{RN}RG$ is free of rank $k$ the action of $RG$ on itself by right multiplication induces an injective ring homomorphism $\Phi:RG\rightarrow M_k(RN)$.
More precisely, $\Phi$ depends on the election of a transversal $\{t_1,...,t_k\}$ of $N$ in $G$ by the following formula for $x = \sum_{g \in G} x_g g \in
\Z G$:
    $$\Phi(x) = \left( \sum_{n\in N} x_{t_i^{-1}nt_j} n \right)_{i,j}\in M_k(RN).$$
Thus
    \begin{equation}\label{TraceFormula}
    \pa{\tr(\Phi(x))}{n^N} = [C_G(n):C_N(n)] \; \pa{x}{n^G}
    \end{equation}
for every $n\in N$.
In particular, $\pa{x}{n^G}\ge 0$ if and only if $\pa{\tr(\Phi(x))}{n^N}\ge 0$.
Moreover, $u\in \V(\Z G,N)$ if and only if $\Phi(u)\in \V_k(\Z N)$.
Therefore, if the Matrix Zassenhaus Problem has a positive solution for $N$ and $u\in \V(\Z G,N)$ then $\pa{u}{n^G}\ge 0$ for every $n\in N$. Moreover, if $N$ is nilpotent then, by Theorem~\ref{HertweckPAdic}, $\pa{u}{g^G}=0$ for every $g\in G\setminus N$. Therefore, if the Matrix Zassenhaus Problem has a positive solution for $N$ nilpotent and $k=[G:N]$ then Sehgal's Problem has a positive solution for $G$ and $N$.

This point of view, taken for the first time in a slightly different formulation in \cite{MarciniakRitterSehgalWeiss1987}, is sometimes called the matrix strategy of Marciniak, Ritter, Sehgal and Weiss. It inspired some research on the Matrix Zassenhaus Problem \cite{MarciniakRitterSehgalWeiss1987, Weiss1988, Weiss1991, LutharPassi1992, CliffWeiss, MarciniakSehgal2000, LeeSehgal2000} and the knowledge of this has also been successfully used to study the question of when $G$ possesses a torsion free complement in $\V(\Z G)$ \cite{MarciniakSehgal2003}. We summarize the results achieved on this question in a theorem.

\begin{theorem}\label{MatrixProblemTh} Let $N$ be a finite nilpotent group.
\begin{enumerate}
\item\label{Problem2Positive} The Matrix Zassenhaus Problem has a positive solution for $N$ and a positive integer $k$ in the following cases:
    \begin{enumerate}
    \item $k = 1$.
    \item If $N$ is abelian and one of the following conditions holds:
        \begin{itemize}
        \item[$\bullet$] $k$ is a prime smaller than any prime divisor of $|N|$ or
        \item[$\bullet$] $k \leq 5$ or
        \item[$\bullet$] for $p$ and $q$ the smallest primes such that $N_p$ and $N_q$ are not cyclic we have $p + q > \frac{k^2 + k - 8}{4}$
        \end{itemize}
    \item $N$ has at most one non-cyclic Sylow subgroup.
    \end{enumerate}
\item If $N$ has more than one non-cyclic Sylow subgroup then the Matrix Zassenhaus Problem has a negative solution for $N$ and some positive integer $k$.
\end{enumerate}
\end{theorem}

\begin{proof}
(1.a) was proved in \cite{Weiss1991}.
The results in (1.b) were proven in \cite[Theorem~4.6]{MarciniakRitterSehgalWeiss1987}, \cite{MarciniakSehgal2000} and  \cite{LeeSehgal2000} respectively. (1.c) and (2) were proved in \cite[Theorem~6.3]{CliffWeiss}.
\end{proof}

The negative result of Cliff and Weiss in Theorem~\ref{MatrixProblemTh}.(2) relies on an induction argument and counterexamples in the minimal finite nilpotent groups with two non-cyclic Sylow subgroups.
More precisely, for these minimal nilpotent groups $N$ they construct concrete characters $\chi$ of the form $\chi_{\alpha}$ for some embedding $\alpha:\GEN{U} \hookrightarrow \GL_k(\Z N)$ which yield to negative solutions to the Matrix Zassenhaus Problem \cite[Section~5]{CliffWeiss}.
This, in their words, "is a counterexample to the strategy of \cite{MarciniakRitterSehgalWeiss1987}".
This statement is only partly true. Indeed, the construction of Cliff and Weiss yields a very general negative solution for the Matrix Zassenhaus Problem in the nilpotent case. But if we focus on the Zassenhaus Conjecture we are only interested in torsion matrices $U$ in $\V_k(\Z N)$ which are the images under the homomorphism $\Phi$ of actual torsion elements of $\V(\Z G, N)$ for some group $G$. However, using \eqref{CharacterPA} one can see that the traces of the matrices $U$ given by Cliff and Weiss have non-vanishing $1$-coefficient and this means that they can not be the images of torsion elements in $\V(\Z G, N)$ for any group $G$ by \eqref{TraceFormula} and the Berman-Higman Theorem (see \cite[Proposition~1.4]{Sehgal1993} or \cite[Proposition~1.5.1]{GRG1}).

\section{The Cliff-Weiss inequalities}\label{SectionInequalities}

The goal of this section is to obtain the inequalities mentioned in the introduction, a particular case of which is described in Proposition~\ref{InequalitiesAbelian}.

In the remainder, we fix a nilpotent normal subgroup $N$ of a finite group $G$ and a torsion element $u$ of $\V(\Z G,N)$.
We also fix the following notation: $k=[G:N]$; $U=\GEN{u}$; $\alpha: U \hookrightarrow \V(\Z G)$ denotes the embedding homomorphism; $\chi = \chi_{\alpha}$, the character of the double action $\Z (U \times G)$-module $\Z[\alpha]$; $\Phi: \Z G \rightarrow M_k(\Z N)$ denotes the homomorphism defined in subsection~\ref{SubsectionMatrixStrategy} for a fixed transversal of $N$ in $G$.
Moreover, for an element $g \in G$ we denote $\pa{u}{g}=\pa{u}{g^G}$, for shortness, and define for $g \in G$ a subgroup of $U \times G$ by
    $$[g] = \langle (u, g) \rangle.$$

Using \eqref{CharacterPA} and \eqref{TraceFormula} we have
\begin{equation}\label{ResRalpha}
\res_{U\times N}^{U\times G}(\chi)=\chi_{\Phi\circ \alpha}.
\end{equation}
That is, the character of the double action $\Z(U\times N)$-module $\Z[\Phi\circ \alpha]$ is the restriction to $U\times N$ of the character of the double action $\Z(U\times G)$-module $\Z[\alpha]$.

Theorem~\ref{HertweckPAdic} implies that the set
		\begin{equation}\label{MDef}
	    \ell_{\mathbb{Q}G}(u) = \{n \in N : u_p \text{ and } n_p \text{ are rationally conjugate for every prime } p\}
		\end{equation}
contains all the elements $g \in G$ with $\pa{u}{g}\ne 0$ and if $n \in \ell_{\mathbb{Q}G}(u)$ then $\ell_{\mathbb{Q}G}(u) = \ell_G(n)$.
This implies the following obvious observation:

\begin{remark}
Suppose that for every $n,m \in N$ with $n^G \ne m^G$ there is a prime integer $p$ with $n_p^G\ne m_p^G$. Then Sehgal's Problem has a positive solution for $G$ and $N$.
\end{remark}

Clearly, $|u|=|n|$ for every $n \in \ell_{\mathbb{Q}G}(u)$.

We first give a description of $\chi$ and $\res_{U\times N}^{U\times G}(\chi)$ in terms of induced characters.

\begin{lemma}\label{CharacterAsIntegralPermutation} We have
\[\chi= \sum_{n^G, \ n \in \ell_{\mathbb{Q}G}(u)} \pa{u}{n} \; \ind_{[n]}^{U \times G}(1)=\sum_{g^G} \pa{u}{g} \; \ind_{[g]}^{U\times G}(1)\]
 and
\[\res_{U\times N}^{U\times G}(\chi) = \sum_{n^N} [C_G(n):C_N(n)] \ \pa{u}{n} \ \ind_{[n]}^{U\times N}(1).\]
\end{lemma}

\begin{proof}
Let $\theta = \sum_{n^G, \ n \in \ell_{\mathbb{Q}G}(u)} \pa{u}{n} \; \ind_{[n]}^{U\times G}(1)$.
Since $\pa{u}{g} \neq 0$ implies $g \in \ell_{\mathbb{Q}G}(u)$, the second equality of the first equation holds.
Moreover, for the same reason $\chi$ and $\theta$ are equal on $(U\times G) \setminus (U \times N)$ by \eqref{CharacterPA1}.
By Mackey decomposition, for every $n \in \ell_{\mathbb{Q}G}(u)$, we have
	\begin{eqnarray*}
	\res_{U\times N}^{U\times G} \left( \ind_{[n]}^{U\times G}(1) \right) &=&
	\sum_{g\in G/N} \ind_{[n^g]}^{U\times N}(1) =
	[\Cen_G(n)N:N]\sum_{m^N\subseteq n^G} \ind_{[m]}^{U\times N}(1) \\
    &=&
	[C_G(n):C_N(n)]\sum_{m^N\subseteq n^G} \ind_{[m]}^{U\times N}(1),
	\end{eqnarray*}
where $\sum_{g\in G/N}$ means a sum with $g$ running on a set of representatives of cosets of $N$ in $G$.
Therefore
	$$\res_{U\times N}^{U\times G}(\theta)=
	\sum_{n^N} [C_G(n):C_N(n)] \ \pa{u}{n} \ \ind_{[n]}^{U\times N}(1).$$

So, in order to prove the lemma, it is enough to prove that $\chi$ and $\theta$ agree on $U\times N$.
Using \eqref{ResRalpha} and \cite[Lemma~4.2 and Proposition~4.3]{CliffWeiss} we have that
	$$\res_{U\times N}^{U\times G}(\chi)=\chi_{\Phi\circ \alpha} =
	\sum_{n^N} a_n \ \ind_{[n]}^{U\times N}(1)$$
for unique integers $a_n$.
Evaluating both sides at $(u,n)$  for $n\in N$ and using \eqref{CharacterPA1} we deduce that
	$$a_n=[C_G(n):C_N(n)] \pa{u}{n}.$$
In particular, if $n$ and $m$ are conjugate in $G$ we have $a_n=a_m$ and if $n\not \in \ell_{\mathbb{Q}G}(u)$ then $a_n=0$. Therefore
	\begin{eqnarray*}
	\res_{U\times N}^{U\times G}(\chi) &=&
	\sum_{n^G} [C_G(n):C_N(n)] \ \pa{u}{n}\sum_{m^N \subseteq n^G} \ \ind_{[m]}^{U\times N}(1) \\
	&=&
	\sum_{n^N} [C_G(n):C_N(n)] \ \pa{u}{n} \ \ind_{[n]}^{U\times N}(1) =
	\res_{U\times N}^{U\times G}(\theta),
	\end{eqnarray*}
as desired.
\end{proof}

A first application of the description of $\res^{U\times G}_{U\times N}(\chi)$ is the following equality which will be used in Section~\ref{SectionApplications}:
\begin{equation}\label{anSums}
\sum_{n^N, n \in N} [C_G(n):C_N(n)]\pa{u}{n} = [G:N].
\end{equation}
Indeed, $\res_{U\times N}^{U\times G}(\chi)$ is a character of degree $|G|$ and each $\ind_{[n]}^{U\times N}(1)$, with $n\in \ell_{\Q G}(u)$, has degree $|N|$. So, \eqref{anSums} follows from the second formula in Lemma \ref{CharacterAsIntegralPermutation}.

The description of $\chi$ provides us with the first kind of inequalities which one might call ``global inequalities''.

If $n \in \ell_{\mathbb{Q}G}(u)$ and $\psi$ is a character of $U \times G$ then let
$$a(n, \psi)=\frac{1}{|u|}\sum_{i=0}^{|u|-1} \psi(u^i, n^i).$$
Observe that if the order of $n$ divides the order of $u$, in particular if $n\in \ell_{\Q G}(u)$, then
$a(n, \psi) = \GEN{1,\res_{[n]}^{U\times G} \psi}_{[n]}\in \Z^{\ge 0}$.

\begin{proposition}\label{EquationCharacter}
If $\psi$ is a character of $U\times G$ then
\[\langle \chi, \psi \rangle = \sum_{n^G, n \in \ell_{\mathbb{Q}G}(u)} a(n,\psi) \ \pa{u}{n} \]
and, in particular,
$$\sum_{n^G, n \in \ell_{\mathbb{Q}G}(u)} a(n,\psi) \ \pa{u}{n} \ge 0.$$
\end{proposition}

\begin{proof}
By Frobenius Reciprocity and Lemma~\ref{CharacterAsIntegralPermutation} we have
	\begin{eqnarray*}
	0&\le& \GEN{\chi,\psi}_{U \times G} = \sum_{n^G, n\in \ell_{\mathbb{Q}G}(u)} \pa{u}{n} \GEN{\ind_{[n]}^{U\times G}(1),\psi}_{U \times G} \\
    &=&
    \sum_{n^G, n \in \ell_{\mathbb{Q}G}(u)} \pa{u}{n} \GEN{1,\res_{[n]}^{U\times G}(\psi)}_{[n]}
	= \sum_{n^G, n \in \ell_{\mathbb{Q}G}(u)} a(n,\psi) \pa{u}{n}.
	\end{eqnarray*}
\end{proof}

Cliff and Weiss \cite{CliffWeiss} obtained more precise information on the character $\res^{U\times G}_{U\times N}(\chi)$ and this allows us to formulate the ``local inequalities'' in Theorem~\ref{Inequalities} below.

Fix a prime integer $p$ and $\x \in \ell_{\mathbb{Q}G}(u_p)$.
Note that $\x$ is unique up to conjugacy in $G$.
For every $n \in N_{p'}$ and every character $\psi$ of $U_{p'} \times N_{p'}$ set
    \begin{eqnarray*}
    a'_{\x}(n,\psi) &=& [\Cen_G(n\x):\Cen_N(n\x)] \; \GEN{1,\res^{U_{p'} \times N_{p'}}_{[n]_{p'}} \psi}_{[n]_{p'}} \text{ and}\\
    a_{\x}(n,\psi) &=& \sum_{g\in \Cen_G(\x)/\Cen_N(\x)} \GEN{1,\res_{[n^g]_{p'}}^{U_{p'} \times N_{p'}}\psi}_{[n^g]_{p'}}.
    \end{eqnarray*}

Consider the following maps
    $$\matriz{{ccccc}
    \Cen_G(\x)/\Cen_N(\x) & \stackrel{\alpha}{\rightarrow} & \Cen_G(\x)N/ N & \stackrel{\beta}{\rightarrow} & \Cen_G(n \x)N \backslash \Cen_G(\x) N \\
    g\Cen_N(\x) & \mapsto & gN & \mapsto & \Cen_G(n\x)Ng}$$
for $g\in \Cen_G(\x)$.
(Caution: $X\backslash Y$ stands for right cosets of $X$ in $Y$. Do not confuse it with set theoretical difference.)
Then $\alpha$ is a group isomorphism and $\beta$ is a surjective map such that $\beta(gN)=\beta(hN)$ if and only if $gh^{-1}\in \Cen_G(n\x)N$ and in this case $n^g$ and $n^h$ are conjugate in $N$.
Thus
    \begin{eqnarray*}
    \GEN{1,\res_{[n^g]_{p'}}^{U_{p'} \times N_{p'}}\psi}_{[n^g]_{p'}}&=&
    \GEN{\ind_{[n^g]_{p'}}^{U_{p'}\times N_{p'}} (1),\psi}_{U_{p'}\times N_{p'}} \\ &=&\GEN{\ind_{[n^h]_{p'}}^{U_{p'}\times N_{p'}} (1),\psi}_{U_{p'}\times N_{p'}} =\GEN{1,\res_{[n^h]_{p'}}^{U_{p'} \times N_{p'}}\psi}_{[n^h]_{p'}}.
    \end{eqnarray*}
Therefore, as $[NC_G(n\x):N]=[C_G(n\x):C_N(n\x)]$, we have
    \begin{eqnarray}\label{Inequality1}
    a_{\x}(n,\psi) &=& \sum_{g\in \Cen_G(\x)/\Cen_N(\x)} \GEN{1,\res_{[n^g]_{p'}}^{U_{p'} \times N_{p'}} \psi}_{[n^g]_{p'}} \nonumber\\
    &=& \sum_{g\in \Cen_G(n\x)N \backslash \Cen_G(\x) N} [\Cen_G(\x)N:N] \GEN{1,\res_{[n^g]_{p'}}^{U_{p'} \times N_{p'}}\psi}_{[n^g]_{p'}} \\
    &=& \sum_{g\in \Cen_G(n\x)N \backslash \Cen_G(\x) N} a'_{\x}(n^g,\psi).\nonumber
    \end{eqnarray}

\begin{theorem}\label{Inequalities}
Let $p$ be a prime integer and $\x \in \ell_{\mathbb{Q}G}(u_p)$. Then for every character $\psi$ of $U_{p'}\times N_{p'}$ we have
    $$\sum_{n^{\Cen_G(\x)}, \ n \in \ell_{\mathbb{Q}G}(u_{p'})} a_{\x}(n,\psi) \pa{u}{n\x} = \sum_{n^N, \ m \in \ell_{\mathbb{Q}G}(u_{p'})} a'_{\x}(n,\psi) \pa{u}{n\x}\ge 0.$$
\end{theorem}

\begin{proof} By \eqref{Inequality1}, for a fix $n \in \ell_{\mathbb{Q}G}(u_{p'})$ we have
    \begin{eqnarray*}
    \sum_{m^{N_{p'}},  m \in n^{\Cen_G(\x)}} a'_{\x}(m,\psi) \; \pa{u}{m\x} &=&
    \left(\sum_{g\in \Cen_G(n\x)N\backslash \Cen_G(\x)N} a'_{\x}(n^g,\psi)\right) \pa{u}{n\x}  \\
    &=& a_{\x}(n,\psi)\; \pa{u}{n\x}.
    \end{eqnarray*}
This implies the equality.
By Lemma~\ref{CharacterAsIntegralPermutation}, we have
	$$\res_{U\times N}^{U\times G}(\chi) =	
    \sum_{y^{N_p}, y \in \ell_{\mathbb{Q}G}(u_p)} \ind_{[y]_p}^{U_p\times N_p}(1) \otimes \chi_y$$
with
    $$\chi_y=\sum_{n^{N_{p'}},n \in \ell_{\mathbb{Q}G}(u_{p'})} [\Cen_G(ny):\Cen_N(ny)] \; \pa{u}{ny} \; \ind_{[n]_{p'}}^{U_{p'}\times N_{p'}}(1).$$
By \cite[Theorem~3.3]{CliffWeiss}, each $\chi_y$ is a proper character of $U_{p'}\times N_{p'}$.
We apply this to $y=\x$.

By Frobenius reciprocity we have
    $\GEN{\ind_{[n]_{p'}}^{U_{p'}\times N_{p'}}(1),\psi}_{U_{p'}\times N_{p'}} =
    \GEN{1,\res_{[n]_{p'}}^{U_{p'} \times N_{p'}}\psi}_{[n]_{p'}}.$
Thus
    $$a'_{\x}(n,\psi)=[\Cen_G(n\x):\Cen_N(n\x)] \;\GEN{\ind_{[n]_{p'}}^{U_{p'}\times N_{p'}}(1),\psi}_{U_{p'}\times N_{p'}}$$
and therefore
    $$\sum_{n \in {N_{p'}},n \in \ell_{\mathbb{Q}G}(u_{p'})} a'_{\x}(n,\psi) \; \pa{u}{n\x}=\GEN{\chi_{\x},\psi}_{U_{p'}\times N_{p'}}\ge 0.$$
This finishes the proof.
\end{proof}

We next formulate Theorem~\ref{Inequalities} in a special case which includes Proposition~\ref{InequalitiesAbelian}.
This will be our main tool for the applications.

\begin{corollary}\label{InequalityCocyclicNAbelian}
Let $p$ be a fixed prime and let $K$ be a normal subgroup of $N_{p'}$ such that $N_{p'}/K$ is cyclic. Let $\x \in \ell_{\mathbb{Q}G}(u_p)$ and  $n \in N_{p'}$. Then
\begin{align*}
	\sum_{m^{\Cen_G(\x)},m\in \ell_{\mathbb{Q}G}(u_{p'})} [\Cen_G(m\x):\Cen_N(m\x)] \;
	|m^{\Cen_G(\x)}\cap nK| \; \pa{u}{m\x}\ge 0.
	\end{align*}
If $N_{p'}$ is abelian then this inequality is equivalent to
    \begin{align*}
    \sum_{m\in nK} |\Cen_G(m\x)| \; \pa{u}{m\x}\ge 0.
    \end{align*}
\end{corollary}

\begin{proof}
We only have to prove the first inequality because it easily implies the second one after multiplying with $|C_N(\x)| = |C_N(m\x)|$ for $m \in \ell_G(n)$.
We claim that we may assume that $n\in \ell_{\mathbb{Q}G}(u_{p'})$. Indeed, if $\ell_{\mathbb{Q}G}(u_{p'}) \cap nK = \emptyset$ then $|m^{C_G(\x)}\cap nK|=0$ for every $m\in \ell_{\mathbb{Q}G}(u_{p'})$ and hence the inequality is trivial. Otherwise we can replace $n$ by an element in $nK\cap \ell_{\mathbb{Q}G}(u_{p'})$.
This proves the claim. So in the remainder of the proof we assume that $n\in \ell_{\mathbb{Q}G}(u_{p'})$. Hence $n$ and $u_{p'}$ have the same order.

Let $H=\GEN{1\times K, (u_{p'},n)}$. As $N_{p'}/K$ is cyclic, $H$ contains the commutator of $U_{p'}\times N_{p'}$ and hence $H$ is normal in $U_{p'}\times N_{p'}$. Furthermore, $a\mapsto (1,a)$ induces a surjective homomorphism $N_{p'}/K\rightarrow (U_{p'}\times N_{p'})/H$. Then $(U_{p'} \times N_{p'})/H$ is cyclic and so there exists a linear character $\psi$ of $U_{p'} \times N_{p'}$ with kernel $H$. Hence, for $m \in \ell_{\mathbb{Q}G}(u_{p'})$ we have
    \[\GEN{1,\res_{[m]_{p'}}^{U_{p'}\times N_{p'}}(\psi)}_{[m]_{p'}} = \begin{cases}1, & \text{if } (u_{p'},m) \in  H; \\ 0, & \text{otherwise}.\end{cases}.\]
Moreover, as $n$ and $u_{p'}$ have the same order, $(u_{p'},\x)\in H$ if and only if $\x\in nK$.
We thus get
    \begin{eqnarray*}
    a_{\x}(m,\psi) &=& [\Cen_G(m\x):C_N(m\x)]\; \left|H\cap \left(\{u_{p'}\}\times m^{\Cen_G(\x)}\right)\right|  \\
    &=& [\Cen_G(m\x):\Cen_N(m\x)] \; |m^{\Cen_G(\x)}\cap nK|
    \end{eqnarray*}
and this implies the corollary by Theorem~\ref{Inequalities}.
\end{proof}

\begin{remark}
Assume that we are given a class function $\varepsilon: G \rightarrow \mathbb{Z}$ such that $\varepsilon$ only takes values in a local $G$-conjugacy class of elements in $N$ and $\sum_{n^G, n \in N} \varepsilon(n) = 1$.
We think the image of $\varepsilon$ as the partial augmentations of a possible unit in $\V(\Z G, N)$.
In other words when we say that $\varepsilon$ satisfies an inequality formulated for partial augmentations of an element $u$ we consider $\varepsilon(g)$ as $\pa{u}{g}$.

If the values of $\varepsilon$ satisfy the inequalities in Theorem~\ref{Inequalities} they also satisfy the inequalities in Proposition~\ref{EquationCharacter}.
\end{remark}

\begin{proof}
We fix a cyclic group $U=\GEN{u}$ of the same order as an element in the local $G$-conjugacy class where $\varepsilon$ is not vanishing.
To prove our claim we introduce the following notation for $m_p\in N_p$ and $\psi$ a character of $U\times G$:
    $$A(m_p,\psi)=\frac{1}{[G:N]}\sum_{m_{p'}^N, m_{p'} \in N_{p'}}  a'_{m_p}(m_{p'},\psi) \varepsilon(m_pm_{p'}).$$
Then
\begin{eqnarray*}
A(m_p,\psi) &=&\frac{1}{[G:N]}\sum_{m_{p'}^N, m_{p'} \in N_{p'}} a'_{m_p}(m_{p'},\psi) \varepsilon(m_pm_{p'}) \\
%&=&\sum_{m_{p'}^N, m_{p'} \in N_{p'}}  \frac{[\Cen_G(m_pm_{p'}):\Cen_N(m_pm_{p'})]}{[G:N]}  \GEN{1,\res_{[m_{p'}]_{p'}}^{U\times G}(\psi)}_{[m_{p'}]_{p'}} \varepsilon(m_pm_{p'})\\
&=&\sum_{m_{p'}^N, m_{p'} \in N_{p'}} \frac{[N:\Cen_N(m_pm_{p'})]}{[G:\Cen_G(m_pm_{p'})]} \GEN{1,\res_{[m_{p'}]_{p'}}^{U\times G}(\psi)}_{[m_{p'}]_{p'}} \varepsilon(m_pm_{p'}) \\
&=&\sum_{m_{p'}^N, m_{p'} \in N_{p'}} \frac{|(m_pm_{p'})^N|}{|(m_pm_{p'})^G|}  \GEN{1,\res_{[m_{p'}]_{p'}}^{U\times G}(\psi)}_{[m_{p'}]_{p'}} \varepsilon(m_pm_{p'})
\end{eqnarray*}
and
\begin{eqnarray*}
\sum_{m^G, m \in \ell_{G}(n)} a(m,\psi) \varepsilon(m)  &=&
\sum_{m^G, m \in \ell_G(n)} \GEN{1,\res_{[m]}^{U\times G}(\psi)}_{[m]} \varepsilon(m) \\
 &=& \sum_{m^N, m \in N} \frac{|m^N|}{|m^G|} \GEN{1,\res_{[m]}^{U\times G}(\psi)}_{[m]} \varepsilon(m) \\
 &=& \sum_{m_p^N, m_p \in N_p} \GEN{1,\res_{[m_p]_p}^{U\times G}(\psi)}_{[m_p]_p} A(m_p,\psi).
\end{eqnarray*}
Therefore, if $\varepsilon$ satisfies the inequalities of Theorem~\ref{Inequalities} then each $A(m_p,\psi)\ge 0$ and hence $\varepsilon$ satisfies the inequalities of Proposition~\ref{EquationCharacter}.
\end{proof}

\section{Subgroup Lattice calculations}\label{SectionLatticeCalculations}

In this section we make some calculations on the lattice of a finite abelian group $A$ which will be applied in the next section.
A reader only interested in units can skip the proofs.

A subgroup $K$ of $A$ is said to be \emph{cocyclic} if $A/K$ is cyclic.
Let $\Cyc(A)$ denote the set of cyclic subgroups of $A$ and let $\Cocyc(A)$ denote the set of cocyclic subgroups of $A$.
By convention, a \emph{maximal cyclic} (respectively, \emph{minimal cocyclic})  subgroup of $A$ is a maximal element of $\Cyc(A)$ (respectively, a minimal element of $\Cocyc(A)$), rather than a cyclic subgroup of $A$ which is maximal (respectively, minimal) as subgroup of $A$.

It is well know that there is an anti-automorphism $\Psi$ of the lattice of subgroups of $A$ such that $A/B\cong \Psi(B)$ and $A/\Psi(B)\cong B$ for every subgroup $B$ of $A$.
More precisely, on the one hand there is an isomorphism $f:A^*\rightarrow A$, where $A^*$ denotes the group of linear characters of $A$, and on the other hand $B\mapsto B^\perp=\{\lambda\in A^*:\lambda(B)=1\}$ defines an anti-isomorphism from the lattice of subgroups of $A$ to the lattice of subgroups of $A^*$ (see e.g. \cite[Problem~2.7]{Isaacs1976}).
Then $\Psi(B)=f(B^{\perp})$ defines an anti-automorphism in the lattice of subgroups of $A$ and we have $A/B\cong (A/B)^* \cong B^{\perp} \cong f(B^{\perp})=\Psi(B)$ and $A/\Psi(B)=A/f(B^{\perp}) \cong A^*/B^{\perp} \cong B^*\cong B$.

Therefore $\Psi$ restricts to an anti-isomorphism of ordered sets between $\Cyc(A)$ and $\Cocyc(A)$.
This yields a duality principal: A property holds for $\Cyc(A)$ if and only if the dual property holds for $\Cocyc(A)$.
The following lemma collects some of the properties satisfied by the cyclic subgroups and the corresponding dual properties for the cocyclic subgroups of an abelian $p$-group.

We use $H\le G$ (respectively, $H<G$) to express that $H$ is a subgroup (respectively, a proper subgroup) of a group $G$.

\begin{lemma}\label{CyclicCocyclic}
Let $p$ be a prime integer and let $A\cong \prod_{i=1}^e C_{p^i}^{l_i}$ with each $l_i$ a non-negative integer. For every $i=1,\dots,e$ let $L_i=\sum_{j\ge i} l_j$ and set $k=L_1=\sum_{i=1}^e l_i$.
\begin{enumerate}
\item\label{CyclicNoV} If $C\in \Cyc(A)$ and $B, \tilde{B} \leq C$ with $|B|=|\tilde{B}|$ then $B=\tilde{B}$.

\item\label{CocyclicNoV} If $ L \in \Cocyc(A)$ and $L \leq K, \tilde{K}\le A$ with $|K|=|\tilde{K}|$ then $K=\tilde{K}$.

\item\label{CyclicAbove} If $B\in \Cyc(A)$ then $\left|\{C\in \Cyc(A): B \leq C \text{ and } [C:B]=p\}\right|$ is either $0$, $p^{k-1}$ or $\frac{p^k-1}{p-1}$, and the latter holds if and only if $B=1$.
	
\item\label{CocyclicBelow} If $K\in \Cocyc(A)$ then $\left|\{L\in \Cocyc(A): L \leq K \text{ and } [K:L]=p\}\right|$ is either $0$, $p^{k-1}$ or $\frac{p^k-1}{p-1}$, and the latter holds if and only if $K=A$.

\item\label{CyclicCocyclicOrder} The number of cocyclic (respectively, cyclic) subgroups of $A$ of order (respectively, index) $p^i$ is $\frac{p^{L_i}-1}{p-1}p^{(L_i-1)(i-1)+\sum_{j<i} l_jj}$.
		
\item\label{CyclicNoMaxCocyclicNoMin} The number of non-minimal cocyclic (respectively, non-maximal cyclic) subgroups of $A$ of order (respectively, index) $p^i$  is
    $$\frac{p^{L_{i+1}}-1}{p-1}p^{(L_{i+1}-1)(i-1)+\sum_{j\le i} l_j(j-1)}=\frac{p^{L_{i+1}}-1}{p-1}p^{(L_i-1)(i-1)+\sum_{j< i} l_j(j-1)}.$$

%\item\label{CyclicCocyclicMaximalAtLeast} $A$ contains at least $p^{k-1}$ maximal cyclic (respectively, cocyclic) subgroups.
\end{enumerate}
\end{lemma}

\begin{proof}
By duality it is enough to prove the statements about cyclic subgroups.

\eqref{CyclicNoV} is well known.

\eqref{CyclicAbove}
The statement is clear if $k=1$, so suppose that $k>1$.
Let $S=\{a\in A:a^p=1\}$.
Assume that $B$ is contained in a cyclic subgroup $C=\GEN{c}$ with $[C:B]=p$ and let $H$ be the unique subgroup of $C$ of order $p$.
Then $S=H\times L$ with $L\cong C_p^{k-1}$, $C\cap L=1$, and every cyclic subgroup $D$ of $A$ containing $B$ with $[D:B]=p$ is contained in $C\times L$.
If $B=1$ then $C\times L$ is elementary abelian and contains $\frac{p^k-1}{p-1}$ subgroups of order $p$.
Otherwise every cyclic subgroup of $C\times L$ with $[D:B]=p$ is necessarily of the form $\GEN{cx}$ with $x\in L$.
Furthermore if $\GEN{cx}=\GEN{cy}$ with $x,y\in L$ then $cy=(cx)^n$ for some $n$. Then $n \equiv 1 \mod |c|$ and hence $n \equiv 1 \mod p$, so that $x=y$. Thus the number of such groups equals $p^{k-1}$.

\eqref{CyclicCocyclicOrder}
The number of elements of order $p^i$ in $A$ is equal to
\[|C_{p^i}^{L_i}\times \prod_{j<i} C_{p^j}^{l_j}|-|C_{p^{i-1}}^{L_i}\times \prod_{j<i} C_{p^j}^{l_j}|=(p^{L_i}-1)p^{L_i(i-1)+\sum_{j<i} l_jj}.\]
So the number of cyclic  subgroups of order $p^i$ is
\[\frac{(p^{L_i}-1)p^{L_i(i-1)+\sum_{j<i} l_jj}}{\varphi(p^i)}=\frac{p^{L_i}-1}{p-1}p^{(L_i-1)(i-1)+\sum_{j<i} l_jj}.\]

\eqref{CyclicNoMaxCocyclicNoMin} The number of elements of order $p^i$ generating a non-maximal cyclic subgroup of order $p^i$ is
\[|C_{p^i}^{L_{i+1}}\times \prod_{j\le i} C_{p^{j-1}}^{l_j}| - |C_{p^{i-1}}^{L_{i+1}}\times \prod_{j\le i} C_{p^{j-1}}^{l_j}| = (p^{L_{i+1}}-1)p^{L_{i+1}(i-1)+\sum_{j\le i} l_j(j-1)}.\]
Thus, the number of non-maximal cyclic subgroups of order $p^i$ is
\[\frac{p^{L_{i+1}}-1}{p-1}p^{(L_{i+1}-1)(i-1)+\sum_{j\le i} l_j(j-1)}.\]
\end{proof}

Recall that the socle $\Soc(A)$ of $A$ is the unique minimal subgroup of $A$ containing all the subgroups of $A$ of prime order.
In the remainder of the section $\pi$ is the set of primes dividing the cardinality of $A$ and we assume that $A_p$ is the direct product of exactly $k_p$ non-trivial subgroups.
This is equivalent to 	
	$$\Soc(A)\cong \prod_{p\in \pi} C_p^{k_p}.$$

For every $K\in \Cocyc(A)$ and every prime $p$ let
$$\alpha_{K,p}=1-\left| \left\{ L\in \Cocyc(A_p) : L < K_p \text{ and } [K_p:L]=p\right\}\right|$$
and
$$\alpha_K= \prod_{p\in \pi} \alpha_{K,p}.$$
For example, if $K_p$ is minimal in $\Cocyc(A_p)$ then $\alpha_{K,p}=1$.
Moreover, by Lemma~\ref{CyclicCocyclic}.\eqref{CocyclicBelow},
\begin{equation}\label{alphaA}
\alpha_{A,p} = 1 - \frac{p^{k_p} - 1}{p -1} = p\frac{1-p^{k_p-1}}{p-1}.
\end{equation}
Thus
\begin{equation}\label{AlphaA}
\alpha_A = \prod_{p\in \pi} p\frac{1-p^{k_p-1}}{p-1}.
\end{equation}

If $A_p$ is cyclic and $K_p\ne 1$ then $\alpha_{K,p}=0$. Moreover, if $K_p$ is minimal in $\Cocyc(A_p)$ then $\alpha_{K,p}=1$. In all the other cases, i.e. when
$A_p$ is non-cyclic and $K_p$ is non-minimal in $\Cocyc(A_p)$, $\alpha_{K,p}$ is negative. This follows from Lemma~\ref{CyclicCocyclic}.
Therefore $\alpha_K<0$ if and only if $K_p = 1$ for all cyclic $A_p$ and the number of primes $p$ for which $A_p$ is non-cyclic and $K_p$ is not minimal
in $\Cocyc(A_p)$ is odd.

\begin{lemma}\label{SumAlphas=1}
If $K\in \Cocyc(A)$ then
	$$\sum_{\substack{L\in \Cocyc(A), \\ L \leq K}} \alpha_L=1.$$
\end{lemma}

\begin{proof}
We start proving the lemma for the case where $A$ is a $p$-group. In this case $\alpha_K = \alpha_{K,p}$ and the lemma is clear if $K$ is minimal in $A$.
Suppose that $K$ is a non-minimal element of $\Cocyc(A)$. Arguing by induction on the cardinality of $K$ we may assume that the lemma holds if $K$ is replaced by any cocyclic subgroup of $A$ properly contained in $K$. Let $L_1,\dots,L_m$ be the cocyclic subgroups $L$ of $A$ with $[K:L]=p$. Thus $\alpha_K=1-m$.
By Lemma~\ref{CyclicCocyclic}.\eqref{CocyclicNoV}, for every cocyclic subgroup $L$ of $A$ properly contained in $K$ there is a unique $i=1,\dots,m$ such that $L \leq L_i$. Thus
$$\sum_{\substack{L\in \Cocyc(A), \\ L \leq K}} \alpha_L=\alpha_K + \sum_{i=1}^m \sum_{\substack{L\in \Cocyc(A), \\ L \leq L_i}} \alpha_L = \alpha_K+m=1.$$
This proves Lemma~\ref{SumAlphas=1} for $A$ a $p$-group.

Assume now that $A$ is arbitrary abelian. The map $L\mapsto (L_p)_p$ defines a bijection from $\Cocyc(A)$ to $\prod_p \Cocyc(A_p)$ and $\alpha_{L,p}=\alpha_{L_p}$ for every $L\in \Cocyc(A)$ and every prime $p$ (where $\alpha_{L_p}$ is defined relative to $A_p$).
Thus
$$\sum_{\substack{L\in \Cocyc(A), \\ L \leq K}} \alpha_L = \sum_{\substack{L\in \Cocyc(A), \\ L \leq K}} \prod_p \alpha_{L,p} = \prod_p \sum_{\substack{L\in \Cocyc(A_p), \\ L \leq K_p}} \alpha_{L} = 1,$$
as desired. This finishes the proof.
\end{proof}

Let
    \begin{eqnarray*}
    \Cocyc^+(A)&=&\{K\in \Cocyc(A) : \alpha_K>0\}\\
    \Cocyc^-(A)&=&\{K\in \Cocyc(A) : \alpha_K<0\}.
    \end{eqnarray*}
Observe that $K\in \Cocyc^+(A)\cup \Cocyc^-(A)$ if and only if $K_p=1$ for every $p$ with $A_p$ cyclic and, in that case,
$K\in \Cocyc^+(A)$ (respectively, $K\in \Cocyc^-(A)$) if and only if the number of primes for which $A_p$ is non-cyclic and $K_p$ is non-minimal is even (respectively, odd).
Set
    $$m^+_A = \sum_{K\in \Cocyc^+(A)} \alpha_K[A:K] \qand m^-_A = \quad -\smashoperator[l]{\sum_{K\in \Cocyc^-(A)}} \alpha_K[A:K]$$
%We also define
%    $$n_A = \begin{cases}
%    m^-_A, & \text{if } A\in \Cocyc^-(A); \\ m^-_A-\alpha_A, & \text{otherwise}.
%    \end{cases}$$
We also define
\begin{equation}\label{mADefinition}
m_A = \begin{cases}
m^-_A, & \text{if } A\in \Cocyc^-(A); \\ m^-_A-\alpha_A, & \text{otherwise}.
\end{cases}
\end{equation}
Clearly $m^+_A>0$, $m^-_A \ge m_A $ and $m^-_A=0$ if and only if $A$ is cyclic. Moreover by Lemma~\ref{SumAlphas=1} we have
\[m_A > - \alpha_A  -\smashoperator[l]{\sum_{K \in \Cocyc^-(A)}}\alpha_K = -1 + \sum_{K \in \Cocyc^+(A)} \alpha_K \]
and in particular $m_A \geq 0$.

\begin{notation}\label{PlusMinusDeltaNotation}
Given a set $S$ we denote by $S^-$ (respectively, $S^+$) the set of subsets of $S$ of odd (respectively, even) cardinality. Moreover, for each subset $X$ of $S$ let $\Delta_X:S\rightarrow \{0,1\}$ denote the characteristic function of $X$.
\end{notation}

\begin{lemma}\label{msGeneral}
The following equalities hold.
	\begin{eqnarray*}
		m^+_A&=&\frac{|A|}{\prod_{p\in \pi} (p-1)p^{k_p-1}} \sum_{X\in \pi^+} \prod_{p\in \pi} (p^{k_p-\Delta_X(p)}-1), \\
		m^-_A&=&\frac{|A|}{\prod_{p\in \pi} (p-1)p^{k_p-1}} \sum_{X\in \pi^-} \prod_{p\in \pi} (p^{k_p-\Delta_X(p)}-1)
	\end{eqnarray*}
\end{lemma}

\begin{proof}
We argue by induction on the cardinality of $\pi$.
Suppose first that $\pi=\{p\}$, i.e. $A$ is a $p$-group.
In this case we adopt the notation of Lemma~\ref{CyclicCocyclic}.
In particular the exponent of $A$ is $p^e$ and $A$ is the product of $k$ non-trivial cyclic subgroups.
Hence $K\in \Cocyc^-(A)$ if and only if $k>1$ and $K$ is a non-minimal cocyclic subgroup of $A$. In that case, if $K \neq A$ then $\alpha_K=1-p^{k-1}$. Furthermore, $\alpha_A=p\frac{1-p^{k-1}}{p-1}$, by \eqref{alphaA}.

Using Lemma~\ref{CyclicCocyclic}.\eqref{CyclicCocyclicOrder} we have
$$m^-_A = p\frac{p^{k-1}-1}{p-1}+ (p^{k-1}-1)\sum_{i=1}^{e-1} p^i \frac{p^{L_{i+1}}-1}{p-1}p^{(L_{i+1}-1)(i-1)+\sum_{j\le i} l_j(j-1)}. $$
Hence
\begin{eqnarray*}
	\frac{p-1}{p^{k-1}-1}m^-_A-p & = &
	\sum_{i=1}^{e-1} (p^{L_{i+1}}-1)p^{i+(L_{i+1}-1)(i-1)+\sum_{j\le i} l_j(j-1)} \\
	& = & p\sum_{i=1}^{e-1} (p^{L_{i+1}}-1)p^{L_{i+1}(i-1)+\sum_{j\le i} l_j(j-1)} \\
	& = & p\left(\sum_{i=1}^{e-1} p^{L_{i+1}i+\sum_{j\le i} l_j(j-1)}-\sum_{i=1}^{e-1} p^{L_{i+1}(i-1)+\sum_{j\le i} l_j(j-1)}\right) \\
	& = & p\left(\sum_{i=2}^e p^{L_i(i-1)+\sum_{j\le i-1} l_j(j-1)}-
	\sum_{i=1}^{e-1} p^{L_{i+1}(i-1)+\sum_{j\le i} l_j(j-1)}\right) \\	
	& = & p\left(p^{\sum_{j\le e} l_j(j-1)}-1\right) = \frac{|A|}{p^{k-1}}-p.
\end{eqnarray*}
For the last equality note that $|A| = p^{\sum_{j\le e} l_jj}$ and $p^k = p^{\sum_{j\le e} l_j}$.
Thus
$$m_A^-=\frac{|A|}{(p-1)p^{k-1}} (p^{k-1}-1)$$
which is the desired equality because $\{p\}$ is the only element of $\pi^-$.

We argue similarly to calculate $m^+_A$.
For that we first calculate, using Lemma~\ref{CyclicCocyclic}, the number of minimal cocyclic subgroups of index $p^i$ in $A$, which is
\begin{eqnarray*}
	n_i&=&
	\frac{p^{L_i}-1}{p-1}p^{(L_i-1)(i-1)+\sum_{j<i} l_jj}-\frac{p^{L_{i+1}}-1}{p-1}p^{(L_i-1)(i-1)+\sum_{j< i} l_j(j-1)} \\
	&=& \frac{1}{p-1}\left(p^{L_i+\sum_{j< i} l_j}-p^{\sum_{j< i} l_j}-p^{L_{i+1}}+1 \right) p^{(L_i-1)(i-1)+\sum_{j< i} l_j(j-1)}.
\end{eqnarray*}
Then
\begin{eqnarray*}
	\frac{p-1}{p}m^+_A &=& (p-1)\sum_{i=1}^e n_i p^{i-1} \\
	&=& \sum_{i=1}^e \left(p^{L_i+\sum_{j< i} l_j}-p^{\sum_{j< i} l_j}-p^{L_{i+1}}+1\right) p^{L_i(i-1)+\sum_{j< i} l_j(j-1)}
\end{eqnarray*}
\begin{eqnarray*}
	&=&
	\sum_{i=1}^e p^{L_ii+\sum_{j< i} l_jj}
	-\sum_{i=1}^e p^{L_i(i-1)+\sum_{j< i} l_jj} \\
	&&
	-\sum_{i=1}^e p^{L_{i+1}i +\sum_{j< i+1} l_j(j-1)}+
	\sum_{i=1}^e p^{L_i(i-1)+\sum_{j< i} l_j(j-1)}.
\end{eqnarray*}
As, for $i=2,\dots,e$, the $(i-1)$-th summand of the first (respectively, third) sum cancels the $i$-th summand of the second (respectively, fourth) sum, we obtain
$$\frac{p-1}{p} m^+_A = p^{L_ee+\sum_{j< e} l_jj}-1-p^{\sum_{j< e+1} l_j(j-1)}+1=|A|-\frac{|A|}{p^k}.$$
Thus $m^+_A=\frac{|A|}{p^{k-1}} \; \frac{p^k-1}{p-1}$, which is the desired equality because $\pi^+=\{\emptyset\}$.
This finishes the first step of the induction argument.

For the induction step observe that for $A=B\times C$, where $B$ and $C$ are Hall subgroups of $A$, we have
\begin{equation}\label{msCoprime}
m^+_{A}=m^+_Bm^+_C+m^-_B m^-_C \qand m^-_A=m^+_Bm^-_C+m^-_Bm^+_C.
\end{equation}
This follows easily by observing that $(K,L)\mapsto K\times L$ defines a bijection $\Cocyc(B)\times \Cocyc(C) \rightarrow \Cocyc(A)$. Moreover for every $K\in \Cocyc(B)$ and $L\in \Cocyc(C)$ we have $[A:K\times L]=[B:K][C:L]$, $\alpha_{K\times L} = \alpha_K \alpha_L$ and $K \times L \in \Cocyc^+(A)$ if and only if $(K,L)\in (\Cocyc^+(B)\times \Cocyc^+(C))\cup (\Cocyc^-(B)\times \Cocyc^-(C))$.
\end{proof}

Recall that the set $A^*$ of linear characters of $A$ is an orthonormal basis of the maps $A\rightarrow \C$ with respect to the hermitian product $\GEN{-,-}_A$. Thus, for every class function $\psi$ of $A$ we have
	$$\psi=\sum_{\varphi\in A^*} \GEN{\varphi,\psi}_A \varphi.$$
For each $K\in \Cocyc(A)$ we fix a linear character $\varphi_K$ of $A$ with kernel $K$ and let $\psi_K$ denote the sum of the linear characters of $A$ with kernel $K$. Moreover, for a positive integer $k$ denote by $\zeta_k$ a complex primitive $k$-th root of unity.
Then
	$$A^* = \{\sigma \circ \phi_K : K\in \Cocyc(A), \sigma\in \Gal(\Q(\zeta_{[A:K]})/\Q) \}$$
and
	$$\psi_K = \sum_{\sigma\in \Gal(\Q(\zeta_{[A:K]})/\Q)} \sigma\circ \varphi_K.$$
Recall that $\Delta_B$ denotes the characteristic function of $B$, which now is going to be used for subsets of $A$. Then the regular representation $\rho$ of $A$ satisfies
	\begin{equation}\label{Regularpsi}
	\rho=|A|\Delta_{\{1\}} = \sum_{K\in \Cocyc(A)} \psi_K.
	\end{equation}

If $B$ is a subgroup of $A$ then
$$\GEN{\varphi_K,\Delta_B}_A  = \frac{1}{|A|}\sum_{b\in B} \varphi_K(b) = \begin{cases} \frac{1}{[A:B]}, & \text{if } B \leq K; \\ 0, & \text{otherwise} \end{cases}$$
Therefore
	\begin{equation}\label{Deltapsi}
	[A:B]\;\Delta_B = \sum_{K/B\in \Cocyc(A/B)} \psi_K.
	\end{equation}

Using \eqref{Regularpsi}, \eqref{Deltapsi} and Lemma~\ref{SumAlphas=1} we deduce
	\begin{eqnarray*}
	\sum_{K\in \Cocyc(A)} \alpha_K[A:K] \Delta_K &=& \sum_{K\in \Cocyc(A)} \alpha_K \sum_{L\in \Cocyc(A), K \leq L} \psi_L \\
	&=& \sum_{L\in \Cocyc(A)} \left(\sum_{K\in \Cocyc(A),K \leq L} \alpha_K\right) \psi_L=\rho.
	\end{eqnarray*}
So
\begin{equation}\label{Delta_a}
\sum_{K\in \Cocyc(A)} \alpha_K[A:K] \Delta_{aK} = |A|\Delta_{ \{a\} }.
\end{equation}

For $a\in A$ let
\begin{equation}\label{Deff}
f_a = \sum_{K\in \Cocyc^+(A)\setminus \{A\}} \alpha_K [A:K]\Delta_{aK} - \sum_{K\in \Cocyc^-(A)\cup \{A\}} \alpha_K [A:K] \Delta_{A\setminus aK}.
\end{equation}
Using $\Delta_{A\setminus aK} = \Delta_A - \Delta_{aK}$ and \eqref{Delta_a} we obtain
	\begin{eqnarray*}
	f_a 	&=& \sum_{K\in \Cocyc(A)} \alpha_K [A:K]\Delta_{aK} - \sum_{K\in \Cocyc^-(A)\cup \{A\}} \alpha_K [A:K] \Delta_{A} \\
&=& |A|\Delta_{ \{a \}} + m_A\Delta_A.
	\end{eqnarray*}
Thus
\begin{equation}\label{f_a}
	f_a(x) = \begin{cases}|A|+m_A, & \text{if } x=a; \\ m_A, & \text{otherwise}.\end{cases}
\end{equation}

\section{Applications}\label{SectionApplications}

\subsection{Positive results on Sehgal's Problem}

Our contribution to Sehgal's Problem is a consequence of Corollary~\ref{InequalityCocyclicNAbelian}, \eqref{anSums} and the formulae obtained in Section~\ref{SectionLatticeCalculations}.

In this subsection $N$ is a normal nilpotent subgroup of a finite group $G$ and $u$ a torsion element of $\V(\Z G, N)$.
We also use Notation~\ref{PlusMinusDeltaNotation}.

\begin{lemma}\label{Bound}
Suppose that $N=A \times B$ where $A$ is an abelian Hall subgroup of $N$ and $B$ has at most one non-cyclic Sylow subgroup.

If $b\in B$ then
	\begin{equation}
	\sum_{a\in A} |\Cen_G(ab)| \; \pa{u}{ab} \ge  0, \label{SumPositive}
    \end{equation}

If $n \in N$ and $\pi$ is the set of prime divisors of $|A|$ then the following inequalities hold where $m_A$ is defined in \eqref{mADefinition}:
    \begin{equation}
	|A| \; |\Cen_G(n)|\; \pa{u}{n} + m_A \sum_{a \in A} |\Cen_G(an_{\pi'})| \; \pa{u}{an_{\pi'}} \ge 0 \label{Sumq},
    \end{equation}
    \begin{equation}
	m_A \; [\Cen_G(n_{\pi'}):C_G(n)] + |A| \; \pa{u}{n}  \ge 0. \label{IndexPlusq}
	\end{equation}
\end{lemma}

\begin{proof}
By assumption $B_{p'}$ is cyclic for some prime $p$ which will be fixed throughout.
If $b_p\not\in \ell_{\Q G}(u)$ then the left side part of \ref{SumPositive} is zero, by Theorem~\ref{HertweckPAdic}, and hence the inequality holds. Suppose otherwise $b_p\in \ell_{\Q G}(u)$.
Then \eqref{SumPositive} is a consequence of Corollary~\ref{InequalityCocyclicNAbelian} for $\x=b_p$, $K=A$ and $n = b_{p'}$.

Let $\nu$ be the set of prime divisors of $B_{p'}$.
If $K\in \Cocyc(A)$ then $K \in \Cocyc(N_{p'})$ and hence, applying Corollary~\ref{InequalityCocyclicNAbelian} with $\x = n_p$ and $n=cn_{\nu}$ for $c\in A$, we obtain
\begin{equation}\label{InequalityAAbelian}
0 \leq \sum_{m\in cn_{\nu}K} |\Cen_G(mn_p)| \; \pa{u}{mn_p} = \sum_{a \in A} \Delta_{cK}(a) \; |\Cen_G(an_{\pi'})| \; \pa{u}{an_{\pi'}}.
\end{equation}
In particular,
\begin{equation}\label{InequalitynpiA}
\sum_{a\in A} \Delta_{n_{\pi} K}(a) \; |\Cen_G(an_{\pi'})| \; \pa{u}{an_{\pi'}} \ge 0
\end{equation}
and summing the inequalities in \eqref{InequalityAAbelian}, for $c$ running on a set of representatives of the cosets of $A$ modulo $K$ not containing $n_{\pi}$ we obtain
\begin{equation}\label{InequalityComplementnpiA}
\sum_{a\in A} \Delta_{A\setminus n_{\pi} K}(a) \; |\Cen_G(an_{\pi'})| \; \pa{u}{an_{\pi'}} \ge 0.
\end{equation}
We have $\alpha_K > 0$ for $K\in \Cocyc^+(A)$ and $\alpha_K < 0$ for $K\in \Cocyc^-(A)$.
So, using \eqref{f_a}, \eqref{Deff}, \eqref{InequalitynpiA}, \eqref{InequalityComplementnpiA} and that $\Delta_{A \setminus A} = \Delta_\emptyset$ is constantly $0$, we deduce that
\begin{eqnarray*}
&&\hspace{-1cm}|A| \; |\Cen_G(n)|\; \pa{u}{n} + m_A \sum_{a \in A} |\Cen_G(an_{\pi'})| \; \pa{u}{an_{\pi'}}	\\
&&= \sum_{a\in A} f_{n_{\pi}}(a) \; |\Cen_G(an_{\pi'})| \; \pa{u}{an_{\pi'}} \\
&&= \sum_{K\in \Cocyc^+(A) \setminus \{A\}} \alpha_K [A:K] \sum_{a \in A} \Delta_{n_{\pi}K}(a) \; |\Cen_G(an_{\pi'})| \; \pa{u}{an_{\pi'}}  -   \\
&& \sum_{K\in \Cocyc^-(A) \cup \{A\}} \alpha_K [A:K] \sum_{a \in A} \Delta_{A\setminus n_{\pi}K}(a) \; |\Cen_G(an_{\pi'})| \; \pa{u}{an_{\pi'}} \ge 0.
\end{eqnarray*}
This proves \eqref{Sumq}.
	
Finally we prove \eqref{IndexPlusq}.
For this we first observe that using \eqref{anSums} and \eqref{SumPositive} for elements $b\in B\setminus n_{\pi'}^G$ we get
\begin{align*}
[G:N] =& \sum_{m^N, m \in N} [C_G(m):C_N(m)]\pa{u}{m} = \sum_{b^B} \sum_{a \in A} [C_G(ba):C_N(ba)] \pa{u}{ba}  \\
\geq& \frac{|n_{\pi'}^G|}{|n_{\pi'}^N|} \sum_{a \in A} [C_G(n_{\pi'}a):C_N(n_{\pi'}a)]\ \pa{u}{n_{\pi'}a} \\
 =& \frac{[G:N]}{[C_G(n_{\pi'}):C_N(n_{\pi'})]} \sum_{a \in A} [C_G(n_{\pi'}a):C_N(n_{\pi'})]\ \pa{u}{n_{\pi'}a}\\
 =& \frac{[G:N]}{|C_G(n_{\pi'})|} \sum_{a \in A} |C_G(n_{\pi'}a)|\ \pa{u}{n_{\pi'}a}.
\end{align*}
Thus, using \eqref{Sumq}, we get
$$m_A\; |\Cen_G(n_{\pi'})| \geq m_A\; \sum_{a \in A} |\Cen_G(n_{\pi'}a)| \; \pa{u}{n_{\pi'}a} \geq - |A| \; |C_G(n)| \; \pa{u}{n},$$
and hence \eqref{IndexPlusq} follows.
\end{proof}

To state our contribution to Sehgal's Problem we introduce the following notation for an abelian group $A$ where $\pi$ and $k_p$ are as in Section~\ref{SectionLatticeCalculations}, i.e. $\pi$ is the set of primes dividing the order of $A$ and $A_p$ is a direct product of $k_p$ non-trivial cyclic $p$-groups (see Lemma~\ref{msGeneral}):
\[n_A^- = \frac{|A|}{m_A^-} = \frac{\prod_{p \in \pi} (p-1)p^{k_p -1}}{\sum_{X \in \pi^-} \prod_{p \in \pi} \left(p^{k_p - \Delta_X(p)} - 1\right)}\]
and
	\[n_A = \begin{cases} n_A^-, & \text{if } |\pi| \text{ is odd}; \\ \frac{|A|n_A^-}{|A|-n_A^-\alpha_A}, & \text{otherwise}.\end{cases}\]
See \eqref{AlphaA} for the value of $\alpha_A$.
Actually, this only makes sense if $A$ is non-cyclic for otherwise $m_A^-=0$.
In case $A$ is cyclic we set $n_A^-=n_A=+\infty$. Otherwise, i.e. if $A$ is non-cyclic, we have
	$$|A|=n_A^-m_A^- = n_A m_A.$$

\begin{theorem}\label{Main}
Let $G$ be a finite group and let $N$ be a nilpotent normal subgroup of $G$ of the form $A \times B$ where $A$ is an abelian Hall subgroup of $N$ and $B$ has at most one non-cyclic Sylow subgroup.
If $[C_G(b):C_G(ab)] < n_A$ for every $a\in A$ and $b\in B$ then every torsion element of $\V(\Z G,N)$ is rationally conjugate to an element of $N$.
\end{theorem}

\begin{proof}
Let $u$ be a torsion element of $\V(\Z G, N)$ and let $g\in G$.
If $g\not\in N$ then $\pa{u}{g}=0$, by Theorem~\ref{HertweckPAdic}.
Otherwise $g=ab$ with $a\in A$ and $b\in B$. If $m_A=0$ then $\pa{u}{g}\ge 0$ by \eqref{IndexPlusq}. Otherwise, $[C_G(b):C_G(g)]< n_A=\frac{|A|}{m_A}$, by assumption.
By \eqref{IndexPlusq} we then obtain
\[\varepsilon_g(u) \geq -\frac{m_A[C_G(b):C_G(g)]}{|A|} > -1.\]
This shows that every partial augmentation of every element in $\V(\Z G,N)$ is non-negative and hence every element of $\V(\Z G,N)$ is rationally conjugate to an element $g$ of $G$. In particular, $\pa{u}{g}\ne 0$ and hence $g\in N$.
\end{proof}

If $A$ is cyclic then the hypothesis of Theorem~\ref{Main} holds trivially, as $n_A=+\infty$.
In this case we obtain at once the following corollary which is also a consequence of the sufficient part of \cite[Theorem 6.3]{CliffWeiss} and
Theorem~\ref{HertweckPAdic}.

\begin{corollary}\label{AtMostOnceNonCyclic}
Let $N$ be a nilpotent normal subgroup of a finite group $G$ such that $N$ has at most one non-cyclic Sylow subgroup.
Then every torsion element of $\V(\Z G,N)$ is rationally conjugate to an element of $N$.
\end{corollary}

Unfortunately, for $A$ non-cyclic, the expression of $n_A$ in Theorem~\ref{Main} is quite technical. More friendly bounds can be obtained in special situations. By Corollary~\ref{AtMostOnceNonCyclic} the next cases of interest are when either $A$ is a non-cyclic $p$-group or when every Sylow subgroup of $A$ is generated by at most two elements. The following corollaries give bounds in these situations.

\begin{corollary}\label{OnePrime}
Let $N$ be a nilpotent normal subgroup of $G$ of the form $A \times B$ where $A$ is an abelian Sylow $p$-subgroup of $N$ and $B$ has at most one non-cyclic Sylow subgroup. If $[C_G(b) : C_G(ab)] < p$ for every $a\in A$ and $b\in B$ then every torsion element of $\V(\Z G,N)$ is rationally conjugate to an element of $N$.
\end{corollary}

\begin{proof}
Suppose that $\Soc(A)=C_p^k$. If $k=1$ then the result follows from Corollary~\ref{AtMostOnceNonCyclic}. Otherwise
	$$p-1 < p-1+\frac{p-1}{p^{k-1}-1}  = n_A=n^-_A=\frac{p^{k-1}(p-1)}{p^{k-1}-1} < p.$$
Therefore, the hypothesis of the corollary is equivalent to the hypothesis of Theorem~\ref{Main}.
\end{proof}

\begin{corollary}\label{SimplifiedBound}
Let $N$ be a nilpotent normal subgroup of $G$ of the form $A \times B$ where $A$ is an abelian Hall subgroup of $N$ and $B$ has at most one non-cyclic Sylow subgroup. Assume moreover that any Sylow subgroup of $A$ is generated by at most two elements.
Let $p$ be the smallest prime dividing $|A|$ and let $k$ be the number of primes dividing $|A|$. If $[C_G(b) : C_G(ab)] < \frac{p}{k}$ for every $a\in A$ and $b\in B$ then every torsion element of $\V(\Z G,N)$ is rationally conjugate to an element of $N$.
\end{corollary}

Before proving Corollary~\ref{SimplifiedBound} we need two more technical lemma.
For a positive integer $k$ set $\Z_k=\{0,1,\dots,k-1\}$.
We now use characteristic functions of subsets of $\Z_k$.

\begin{lemma}\label{NumericalLemma}
If $x$ is a positive real number and $k$ is a positive integer then $$\sum_{X\in \Z_k^-} \prod_{i=0}^{k-1} (x+2i+1)^{\Delta_{\Z_k\setminus X}(i)} = k
\prod_{i=1}^{k-1} (x+2i).$$
\end{lemma}

\begin{proof}
Let $A_k^± =\sum_{X\in \Z_k^\pm} \prod_{i=0}^{k-1} (x+2i+1)^{\Delta_{\Z_k\setminus X}(i)}$.
We first prove by induction on $k$ that
\begin{equation}\label{An+-}
kA_k^+=(x+k)A_k^-.
\end{equation}
The case $k=1$ is easy because $A_1^+=x+1$ and $A_1^-=1$, as by convention an empty product equals 1.
For the induction step we will use the following recursive formulae:
$$A_{k+1}^+ = A_k^+(x+2k+1)+A_k^- \qand A_{k+1}^- = A_k^-(x+2k+1)+A_k^+$$
Thus, if we assume that $kA_k^+=(x+k)A_k^-$ then we have
$$\frac{A_{k+1}^+}{A_{k+1}^-} = \frac{A_k^+(x+2k+1)+A_k^-}{A_k^-(x+2k+1)+A_k^+}
= \frac{(x+k)(x+2k+1)+k}{k(x+2k+1)+(x+k)} = \frac{x+k+1}{k+1}.$$
This finishes the proof of \eqref{An+-}.

Now we define $B_k=\prod_{i=1}^{k-1} (x+2i)$. We have to show that $A_k^-=kB_k$ and again we argue by induction on $k$. The case $k=1$ is trivial because $A_1^-=B_1=1$.
Hence we assume $A_k^-=kB_k$ and, using \eqref{An+-} and the recursive formulae for $A_k^±$, we have
    \begin{eqnarray*}
    A_{k+1}^- &=& A_k^-(x+2k+1)+A_k^+ = A_k^-\left(x+2k+1+\frac{x+k}{k}\right) \\
    &=& B_k(k(x+2k+1)+x+k) = (k+1)B_k(x+2k) = (k+1)B_{k+1}.
    \end{eqnarray*}
This finishes the proof of the lemma.
\end{proof}

\begin{lemma}\label{MinimalPrimenA}
Let $A$ be finite abelian group of odd order such that each Sylow subgroup of $A$ is the product of two non-trivial cyclic subgroups. Let $k$ be the number of non-cyclic Sylow subgroups of $A$ and let $p_0$ be the smallest prime with $A_{p_0}$ non-cyclic. Then $\frac{p_0}{k} \leq n_A$.
\end{lemma}

\begin{proof}
Observe that if $C$ is cyclic of order coprime with $|A|$ then $n_A = n_{A \times C}$.
Hence we may assume without loss of generality that every Sylow subgroup of $A$ is non-cyclic.

Let $p_0<p_1<\dots<p_{k-1}$ be the primes dividing $|A|$, so by assumption $\Soc(A_i)\cong C_{p_i}^2$ for every $i$.
Thus
    \begin{eqnarray*}
n_A \geq n_A^- = \frac{\prod_{i=0}^{k-1} (p_i-1)p_i}{\sum_{X\in \Z_k^-} \prod_{i=0}^{k-1} (p_i^{2-\Delta_X(i)}-1)}
    &=&\frac{\prod_{i=0}^{k-1} p_i}{\sum_{X\in \Z_k^-} \prod_{i=0}^{k-1} (p_i+1)^{\Delta_{\Z_k\setminus X}(i)}}.
    \end{eqnarray*}
If $k=1$ then the latter number is $p_0$ which yields the result in this case.

So suppose $k\ge 2$. Then $p_0+2i\le p_i$ for each $i$.
Moreover, if $0< x\le y$ then $\frac{x}{x+1}\le \frac{y}{y+1}$. Hence, using Lemma~\ref{NumericalLemma} for $x=p_0$, we obtain
    \begin{eqnarray*}
    n_A &\ge& \frac{1}{\sum_{X\in \Z_k^-} \prod_{i=0}^{k-1} \frac{(p_i+1)^{\Delta_{\Z_k\setminus X}(i)}}{p_i}} \ge
    \frac{1}{\sum_{X\in \Z_k^-} \prod_{i=0}^{k-1} \frac{(p_0+2i+1)^{\Delta_{\Z_k\setminus X}(i)}}{p_0+2i}} \\
    &=& p_0 \frac{\prod_{i=1}^{k-1} (p_0+2i)}{\sum_{X\in \Z_k^-} \prod_{i=0}^{k	-1} (p_0+2i+1)^{\Delta_{\Z_k\setminus X}(i)}} = \frac{p_0}{k},
    \end{eqnarray*}
as desired.
\end{proof}

We are ready for the
\medskip

\begin{proofof}\textit{Corollary~\ref{SimplifiedBound}}.
By Corollary~\ref{OnePrime} we may assume that $k\ge 2$. Then the order of $A$ is odd because otherwise $[\Cen_G(b):\Cen_G(ab)] < \frac{2}{k}\le 1$, which is not possible, for $a\in A$ and $b\in B$. Then the corollary follows at once from Theorem~\ref{Main} and Lemma~\ref{MinimalPrimenA}.
\end{proofof}

\subsection{Positive results for the Zassenhaus Conjecture}

We finish by proving the Zassenhaus Conjecture for some groups.
To this end, as well as our results on Sehgal's Problem, we need to obtain results also on units not
lying in $\V(\Z G, N)$. A weaker version of the following lemma appeared in \cite{MarciniakRitterSehgalWeiss1987}.

\begin{lemma}\label{Not1ModN}
Let $N$ be a nilpotent normal subgroup of $G$ such that $[G:N]$ is prime. Then any torsion unit in $\V(\Z G)$ which does not lie in $\V(\Z
G,N)$ is rationally conjugate to an element in $G$.
\end{lemma}

\begin{proof}
We argue by induction on the order of $N$, with the case $N=1$ being trivial.
In particular, we assume that if $q$ is a prime dividing the order of $N$ then any torsion unit of $\V(\Z(G/N_q))$ which is not in $\V(\Z(G/N_q),N/N_q)$ is rationally conjugate to an element of $G/N_q$.
Let $p = [G:N]$ and let $u \in \V(\Z G)$ be a torsion unit such that $u \notin \V(\Z G, N)$.
This implies that $p$ divides the order of $u$.
Note that if $N$ is a $p$-group, so is $G$ and the Zassenhaus Conjecture holds by \cite{Weiss1988}.
Let $q$ be a prime dividing $|G|$ different from $p$.
Use bar notation for reduction modulo $N_q$ and its $\Z$-linear extension.
By the induction hypothesis $\overline{u}$ is rationally conjugate to an element of
$\overline{G}$ and hence the partial augmentations of the powers of $\overline{u}$ are non-negative.

If $q$ is not a divisor of the order of $u$ then for every $d\mid n$ the partial augmentations of $\overline{u}^d$ and $u^d$ coincides and hence $u$ is rationally conjugate to an element of $G$.
Hence we may assume that $q$ divides the order of $u$.
By Theorem~\ref{HertweckPAdic}, there exists a non-trivial element $y \in N_q$ such that $y$ is conjugate in $\Z_p G$ to $u_q$.
Then, by \cite[Lemma~2.2]{Hertweck2008}, if $\pa{u^d}{g} \neq 0$ then $y^d$ is conjugate in $G$ to $g_q$.
Since the image of $u$ in $\V(\Z(G/N))$ is non-trivial there exists an element $x \in G$ such that $\pa{u}{x} \neq 0$ and the image of $x$ in $G/N$ is not the identity.
This implies, in particular, that $G = \langle N, x_p \rangle$.
Moreover, by the above, we can assume $x_q = y$, so $x_p \in C_G(y)$.
Let $H = \langle N_{q'},x_p \rangle$. Then $G = HN_q$ and $H \leq C_G(y)$.
We will show that whenever $g_1$ and $g_2$ are non-conjugate elements in $G$ which satisfy $\pa{u^d}{g_1}
\neq 0$ and  $\pa{u^d}{g_2} \neq 0$ then $\overline{g_1}$ and $\overline{g_2}$ are also not conjugate. This will imply that $\pa{u^d}{g}=\pa{\overline{u}^d}{\overline{g}}\ge 0$ for every $g\in G$ and hence $u$ is rationally conjugate to an element of $G$.

So let $g_1$ and $g_2$ be non-conjugate elements in $G$ such that $\pa{u^d}{g_1} \neq 0$ and $\pa{u^d}{g_2} \neq 0$.
By the above, we may assume $g_1 = h_1y^d$ and $g_2 = h_2y^d$ with $h_1, h_2 \in H$.
Assume that $\overline{g_1}$ and $\overline{g_2}$ are conjugate in $G/N_q$.
So there exists also a $z \in H$ such that $h_1^z = h_2$. Then $g_1^z=(h_1y^d)^z = h_2y^d=g_2$,
since $H \leq C_G(y)$. So $g_1$ and $g_2$ are conjugate, contradicting our assumption.
\end{proof}

Combining this lemma with Corollary~\ref{AtMostOnceNonCyclic} we immediately obtain Theorem~\ref{TheoremZC1}.
%Similarly, using our results on Sehgal's Problem we get the following.
Finally we also prove Theorem~\ref{TheoremZC2}
\medskip

\begin{proofof}\textit{Theorem~\ref{TheoremZC2}}.
Let $u$ be a torsion element of $\V(\Z G)$. If $u$ does not lie in $\V(\Z G, N)$ then it is rationally conjugate to an element in $G$ by Lemma~\ref{Not1ModN}. So suppose $u \in \V(\Z G,N)$. If $N$ has at most one non-cyclic Sylow subgroup then $u$ is rationally conjugate to an element in $N$ by Corollary~\ref{AtMostOnceNonCyclic}. Otherwise we can assume that $A$ is an abelian Sylow $p$-subgroup of $N$. We then have $[C_G(b):C_G(ba)] \leq [G:N] < p$ for every $b \in B$ and every $a \in A$. So $u$ is rationally conjugate to an element in $N$ by Corollary~\ref{OnePrime}.
\end{proofof}

\bibliographystyle{amsalpha}
\bibliography{CW}

\noindent
Departamento de Matemáticas, Universidad de Murcia, 30100 Murcia, Spain\newline
email: leo.margolis@um.es, adelrio@um.es

\end{document}